\documentclass[12pt]{article}
\usepackage[utf8]{inputenc}	% Para caracteres en español
\usepackage{amsmath,amsthm,amsfonts,amssymb,amscd}
\usepackage{amsmath}
\usepackage{multirow,booktabs}
\usepackage{graphicx}
\usepackage[table]{xcolor}
\usepackage{fullpage}
\usepackage{lastpage}
\usepackage{enumitem}
\usepackage{fancyhdr}
\usepackage{mathrsfs}
\usepackage{wrapfig}
\usepackage{setspace}
\usepackage{calc}
\usepackage{multicol}
\usepackage{cancel}
\usepackage[margin=3cm]{geometry}
\usepackage{amsmath}
\usepackage{amsrefs}

\usepackage{empheq}
\usepackage{framed}
\usepackage[most]{tcolorbox}
\usepackage{xcolor}
\colorlet{shadecolor}{orange!15}
\newcommand{\bibliofont}{\footnotesize}
\theoremstyle{definition}
\newtheorem{defn}{Definition}[section]
\newtheorem{remark}[defn]{Remark}

\newtheorem{question}{Question}

\theoremstyle{theorem}
\newtheorem{prop}[defn]{Proposition}
\newtheorem{cor}[defn]{Corollary}

\newtheorem*{thm*}{Theorem}
\newtheorem*{lem*}{Lemma}
\newtheorem{lem}[defn]{Lemma}

\begin{document}

\thispagestyle{empty}
%\title{Minimal genus four manifolds}
%\author{Rom\'an Aranda}
%\date{}
  %\email{jose-arandacuevas@uiowa.edu} 
%\maketitle
\begin{center}\Large{Minimal genus four manifolds}\\
\large{Rom\'an Aranda}
\\
{January 2019}
\end{center}
\begin{abstract} 
In 2018, M. Chu and S. Tillmann gave a lower bound for the trisection genus of a closed 4-manifold in terms of the Euler characteristic of $M$ and the rank of its fundamental group. 
We show that given a group $G$, there exist a 4-manifold $M$ with fundamental group $G$ with trisection genus achieving Chu-Tillmann's lower bound. 
\end{abstract}

\setcounter{section}{0}
\section{Introduction} 
Let $M$ be a smooth, oriented, closed 4-manifold. D. Gay and R. Kirby showed in \cite{trisecting_four_mans} that $M$ can be written as the union of three 4-dimensional 1-handlebodies\footnote{A 1-handlebody of dimension four is a 4-ball with 1-handles attached along its boundary; i.e. $\natural_k S^1\times B^3$.} of genus $k_1$, $k_2$ and $k_3$, respectively; with pairwise intersections 3-dimensional connected handlebodies and triple intersection a connected surface of genus $g$. This is called a $(g;k_1,k_2,k_3)-$trisection of $M$. The trisection genus of a 4-manifold $M$, denoted by $g(M)$, is the smallest $g$ such that $M$ admits a $(g;*,*,*)-$trisection. 

Let $\tau$ be a trisection of $M$; i.e.  $M=X_1 \cup X_2\cup X_3$. Each 1-handlebody $X_i$ can be used to span the fundamental group and the first homology of $M$. In \cite{Tillmann_Chu}, M. Chu and S. Tillmann used this to give a lower bound to the trisection genus of a closed 4-manifold. They showed that if $M$ admits a $(g;k_1,k_2,k_3)-$trisection then 
\[ g \geq \chi(M) -2 + 3\text{rk}(\pi_1(M)). \] 

\begin{question} [From \cite{Tillmann_Chu}] \label{Q1}
Given any finitely presented group $G$, is there a smooth, oriented closed 4-manifold $M$ with $\pi_1(M)=G$ such that $g(M)=\chi(M) -2 + 3\text{rk}(G)$? 
\end{question}
In this note, we will give a positive answer of Question \ref{Q1}. 

Fix a finitely presented group $G$. We divide the answer in two steps as follows: \textit{(1)} We state an equivalent version of Question \ref{Q1} in terms of Kirby diagrams of closed 4-manifolds, which translates into a knot theory problem; and then \textit{(2)} we show the existence of a special type of links. 

Along this note, all 4-manifolds will be smooth, oriented and compact. For a link $L$ in a 3-manifold $Y$, $t_Y(L)$ will denote the tunnel number of $L$ in $Y$. We will omit the sub-index $Y$ if there is no confusion of the ambient manifold. 

\section{The proof}  %\subsection{Equivalent statements} 
Let $M$ be a closed 4-manifold and $\tau$ be a $(g;k_1,k_2,k_3)-$trisection of $M$. In \cite{Tillmann_Chu}, M. Chu and S. Tillmann showed the inequality 
\begin{equation} \label{ineq} 
g \geq \chi(M) -2 + 3\text{rk}(\pi_1(M))
\end{equation}
Since $\chi(M)=2+g-(k_1+k_2+k_3)$ and $k_i \geq \text{rk}(\pi_1(M))$ for $i=1,2,3$, equality of Equation (\ref{ineq}) is equivalent to 
\begin{equation} \label{equality}
k_i = rk\left(\pi_1(M)\right) \quad \quad  i=1,2,3
\end{equation}  

The following lemma is an application of Lemmas 13 and 14 of \cite{trisecting_four_mans}. It shows that Question \ref{Q1} is equivalent to finding links with the correct homotopy class and ``small" tunnel number. 

\begin{lem}\label{equivalent_versions}
Let $G\neq 1$ be a finitely presented group of rank $n$. The following are equivalent: 
\begin{enumerate}[label = (\alph*)] 
\item  There is a smooth, oriented closed 4-manifold $M$ with $\pi_1(M)=G$ such that $$g(M)=\chi(M) -2 + 3\text{rk}(G)$$ 
\item There is a smooth, oriented closed 4-manifold $M$ with $\pi_1(M)=G$ having a handle decomposition with one 0-handle, $n$ 1-handles, $j$ 2-handles, $n$ 3-handles and one 4-handle such that the attaching region of the 2-handles is a link in $\#_n S^1\times S^2$ of tunnel number $n+j-1$.
\item There is a link $L$ in $\#_n S^1\times S^2$ of tunnel number at most $n+|L|-1$ such that the link $L$, thought of as a set of homotopy classes of loops based at a point, gives relations for a rank $n$ group presentation for $G$.
\end{enumerate} 
\end{lem}

\begin{proof} 
\textbf{(a) $\Rightarrow$ (b).} Let $M$ be a 4-manifold with $\pi_1(M)=G$ such that $g(M)=\chi(M) -2 + 3\text{rk}(G)$. Let $\tau$ be a $(g;k_1,k_2,k_3)-$trisection with $g=g(M)$. In particular, $\tau$ satisfies Equation (\ref{equality}); i.e., $\tau$ is a $\left(g;n,n,n\right)-$trisection of $M$. Lemma 13 of \cite{trisecting_four_mans} asserts that $M$ admits a handle decomposition with one 0-handle, $n$ 1-handles, $g-n$ 2-handles, $n$ 3-handles and one 4-handle such that the attaching region of the 2-handles is a framed link $L$ contained in the core of one of the handlebodies of a genus $g$ Heegaard splitting for $\#_n S^1\times S^2$. In particular, $t(L) \leq g-1$. The latter must be an equality since Proposition 4.2 of \cite{class_trisections} states that $M$ admits a new $(t(L)+1; n, t(L)+1-|L|,n)$-trisection and $t(L) < g-1$ will give us $t(L)+1-|L|<(g-1)+1-(g-n)=n$; contradicting the fact that $k_2\geq n$. Hence $t(L) = g-1$ and, (b) holds taking $j=g-n$. \\
%Lemma 13 of \cite{trisecting_four_mans} asserts that $M$ admits a handle decomposition like in (b). \\
\textbf{(b) $\Rightarrow$ (a).} %Lemma 14 of \cite{trisecting_four_mans} asserts that if (b) holds then there is a closed 4-manifold $M$ admitting a $(g;n,n,n)-$trisection. In particular Equation (\ref{equality}) holds and Equation (\ref{ineq}) becomes an equality, thus (a). \\ 
Let $M$ be a 4-manifold satisfying (b) with the given handlebody decomposition. By taking the tubular neighborhood of the attaching region of the 2-handles and the tunnels, we obtain a Heegaard surface for $\#_n S^1\times S^2$ of genus $g=n-j$ satisfying the assumptions of Lemma 14 of \cite{trisecting_four_mans}. By the lemma, $M$ admits a $(g;n,n,n)$-trisection. In particular, Equation (\ref{equality}) holds and Equation (\ref{ineq}) becomes an equality, thus (a). \\ 
\textbf{(b) $\Rightarrow$ (c).} Take $L$ to be the attaching region of the 2-handles. \\ 
\textbf{(c) $\Rightarrow$ (b).} Let $L$ be such link and let $x \in \mathbb{Z}^{|L|}$ be any fixed vector. Consider $N$ to be the smooth 4-manifold with a handle decomposition given by one 0-handle, $n$ 1-handles and $|L|$ 2-handles attached along $L$ with framing $x$ with respect to the blackboard. Take $M$ to be the double of $N$. By assumption, $\pi_1(N) \cong G$ and so $\pi_1(M) \cong G$. Turning the handle decomposition of $N$ upside-down gives $M$ a handle decomposition with one 0-handle, $n$ 1-handles, $2|L|$ 2-handles, $n$ 3-handles and one 4-handle, where half of the 2-handles are attached along $L$ and for each component of $L$ there is a 0-framed unknot linked once along the given component and unlinked from the rest of the diagram. \\
Let $\widehat{L} = L \cup L'$ be the attaching region for the 2-handles of $M$. Notice that $$t(\widehat{L}) \leq t(L) + |L| = n + |\widehat{L}| -1.$$
By a version of Lemma 14 of \cite{trisecting_four_mans} for unbalanced trisections (see Proposition 4.2 of \cite{class_trisections}), $M$ admits a $(t(\widehat{L})+1; n, t(\widehat{L})+1 - |\widehat{L}|, n)-$trisection. Using Equation (\ref{ineq}) with $k_2 =  t(\widehat{L})+1 - |\widehat{L}|$ we obtain
\begin{align*} 
n\quad\leq &\quad t(\widehat{L})+1 - |\widehat{L}| \\ 
\quad\leq & \quad \left( t(L) + |L|\right) + 1 - 2|L|\\ 
\quad= & \quad t(L) + 1 -|L| \\
\quad\leq & \quad n.
\end{align*}
Thus, $t(\widehat{L})= n + |\widehat{L}| -1$ and $M$ is the desired 4-manifold. Hence (b).  
\end{proof}  

%\subsection{A construction} 
\begin{remark}
We have shown in Lemma \ref{equivalent_versions} that to answer Question \ref{Q1} in the positive, is enough to find a link $L$ in $\#_n S^1\times S^2$ with tunnel number at most $n+|L|-1$ such that the homotopy classes of the components of $L$, together, read a rank $n$ presentation for the given group $G$. The closed 4-manifold answering Question \ref{Q1} will be the double of a 4-manifold with a Kirby diagram with $n$ 1-handles and 2-handles attached along $L$ with any framing. 
\end{remark} 
The following proposition shows how to build links satisfying (c) in Lemma \ref{equivalent_versions}.
\begin{prop} \label{existence_link}
Let $G$ be a finitely presented group of rank $n$ with a presentation $\langle X|R \rangle$. There exists an $|R|$-component link in $\#_{|X|} S^1\times S^2$ with tunnel number at most $|X|+|R|-1$ such that the words read by the components of $L$ in $\pi_1(\#_{|X|} S^1\times S^2, \star)$ agree with the words in $R$. 
\end{prop}

\begin{cor} \label{answer_Q1}
Given a finitely presented group $G$, there is a closed 4-manifold with fundamental group isomorphic to $G$ satisfying $g(M)=\chi(M) -2 + 3\text{rk}(G)$.
\end{cor}

\begin{proof} [Proof of Proposition \ref{existence_link}]
Take an unknotted graph in $\#_n S^1 \times S^2$ made by $n$ loops, denoted by $\Gamma_0$, generating $\pi_1(\#_n S^1\times S^2, \star)$ and $|R|$ unknotted loops $c_1, \dots, c_{|R|}$ around a neighborhood of $\star$; see Figure \ref{unknotted}.
For each relation $r_i \in R$, take the $i$-th unknotted circle $c_i$ and slide one of its ends along the loops on $\Gamma_0$ so that $c_i$ now reads the word $r_i$ as an element of $\pi_1(\#_n S^1\times S^2, \star)$. Each component of $L$ will be a circle $c_i$, and the rest of the graph can be homotoped to be a system of tunnels for $L$. \\
We have described how to build a link $L$ in $\#_n S^1\times S^2$ with $t(L) \leq n + |L| -1$. 
\end{proof}

\begin{figure}[h]
\centering
\includegraphics[scale=.1]{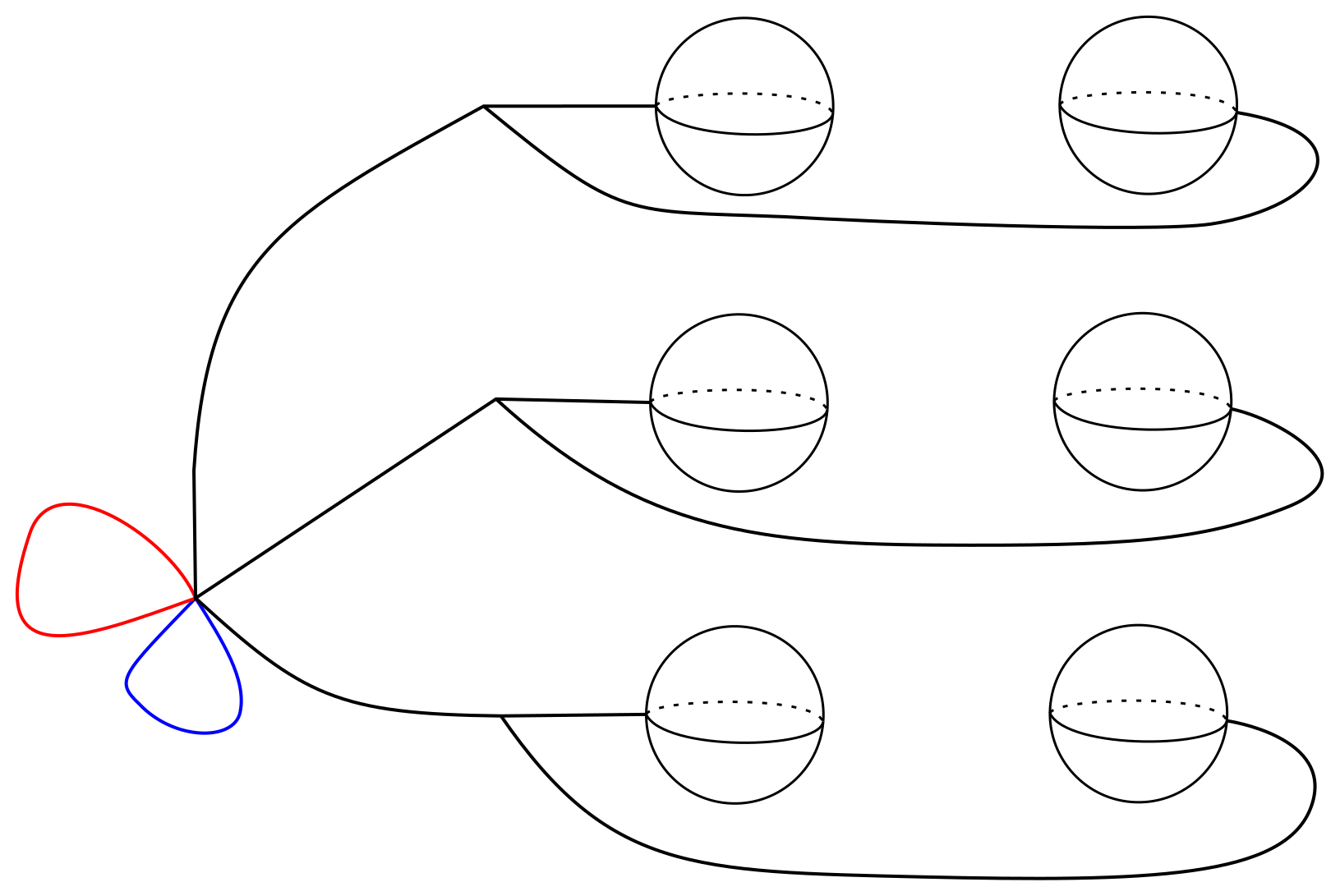}
\caption{Unknotted graph in $\#_3 S^1\times S^2$. The red and blue loops are $c_1$ and $c_2$, respectively.}
\label{unknotted}
\end{figure}

%\subsection{Example} 
%\newpage

The following figures ilustrate the construction of $L$ for the group $G=\langle x,y,z | x^3y^{-2},  [y,z] \rangle$. 
\begin{figure}[h]
\centering
\includegraphics[scale=.22]{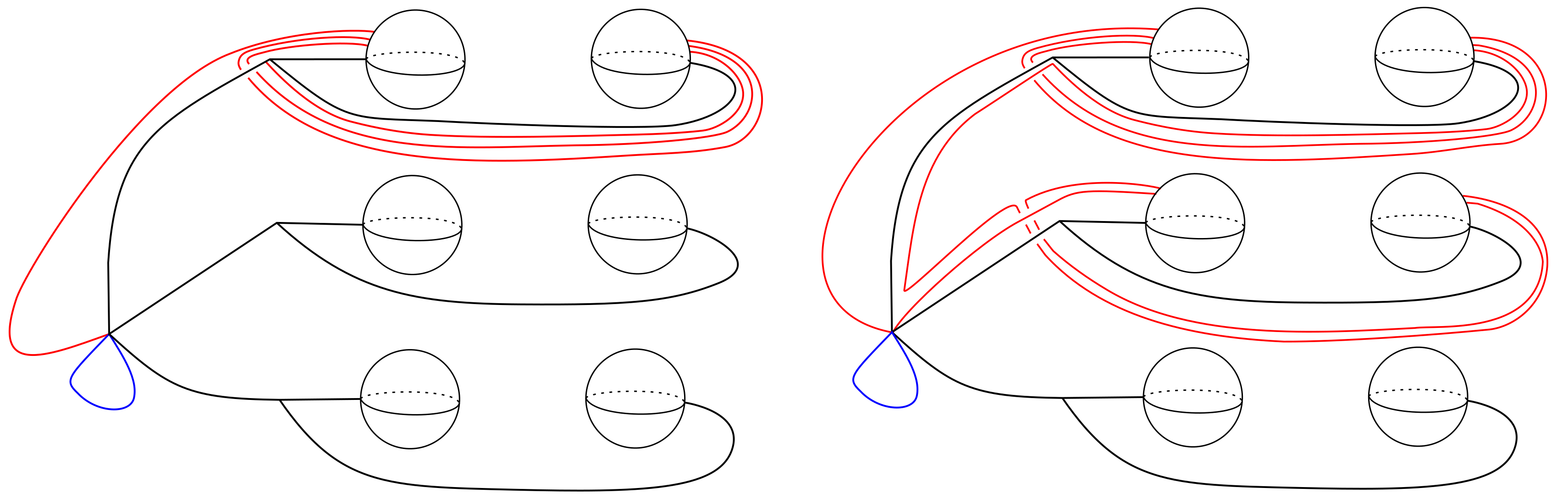}
\caption{Sliding the end of $c_1$ (red) to produce the word $x^3y^{-2}$.}
\label{Step_1}
\end{figure}

\begin{figure}[h]
\centering
\includegraphics[scale=.11]{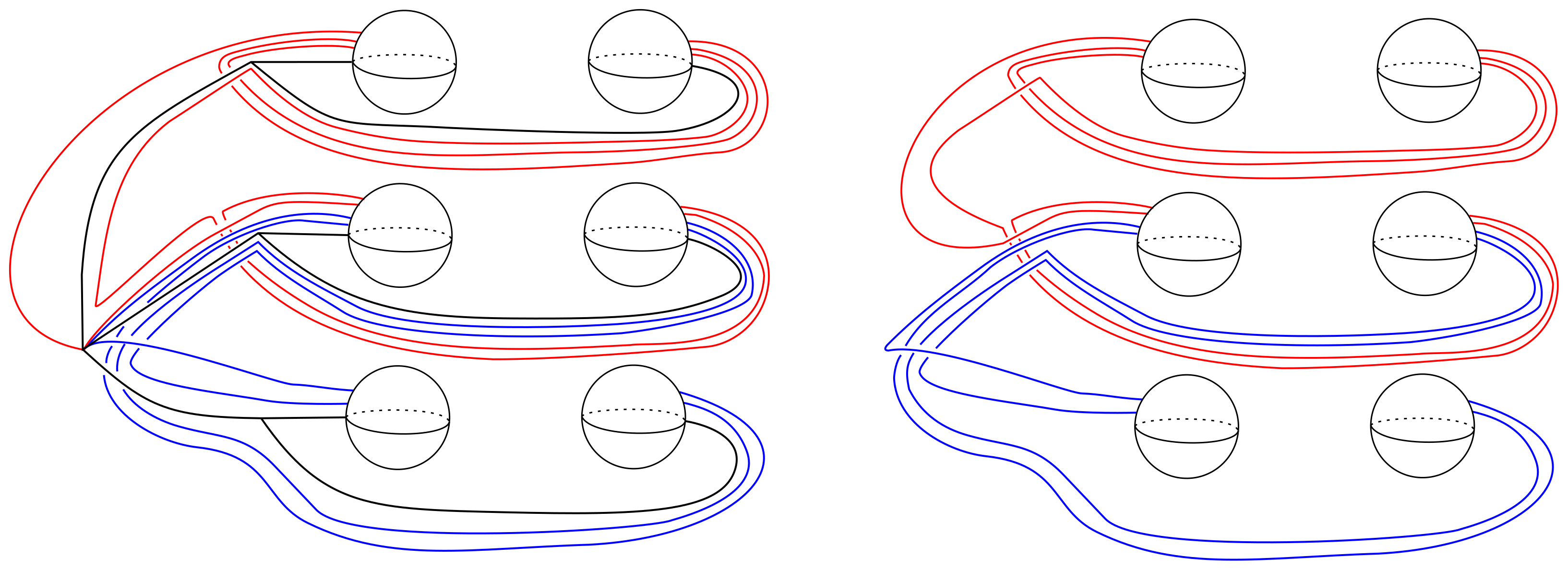}
\caption{Left: The graph after the homotopies, red loop reads $x^3y^{-2}$ and blue loop reads $[y,z]$. Right: The desired link $L$ for $G=\langle x,y,z | x^3y^{-2},  [y,z] \rangle$.}
\label{Step_2}
\end{figure}

\begin{remark} 
Note that the knot type of the link $L$ doesn't change the final 4-manifold built in Corollary \ref{answer_Q1}. This happens since, when taking double of the 4-manifold, the 0-framed unknots around each component of $L$ allow us to slide the handlebody 2-handles and so to change the crossings of the link. The same reason explains why we only care about the framing of the components of $L$ modulo $2$. In any case, one can consider connected sums of our manifolds with copies of $S^2\times S^2$ and $\pm CP^2$ to build infinitely many 4-manifolds solving Question \ref{Q1}.  

%To prove Proposition \ref{existence_link}, we never required the presentation $\langle X|R\rangle$ of $G$ to have the least number of relations $R$ with respect to $|X|$. Thus, presentations with redundant relations will still give 4-manifolds with satisfying Question \ref{Q1}. It would be interesting to relate these 4-manifolds for a given group $G$. 

One of the two 4-manifolds appearing when running the construction with finite cyclic groups $\langle x|x^m = 0\rangle$ is the spun lens space $\mathcal{S}_m$ trisected by J. Meier in \cite{spun_trisections}. One can check this by comparing the Kirby diagram we construct with the one drawn by J. Montesinos in \cite{Heegaard_splittings_of_4_mans}.
\end{remark}

\begin{flushright}
\texttt{email: jose-arandacuevas@uiowa.edu}
\end{flushright}
\end{document}